\documentclass[10pt,reqno]{article}
\usepackage{latexsym, amsmath, amssymb, amsthm, a4, epsfig}
\usepackage{graphicx}
\usepackage{ascmac}

\newtheorem{theorem}{Theorem}[section]
\newtheorem{cor}[theorem]{Corollary}
\newtheorem{lemma}[theorem]{Lemma}
\newtheorem{prop}[theorem]{Proposition}

\newcommand{\p}{\partial}

\newcommand{\eqnref}[1]{(\ref {#1})}

\newcommand{\Nbb}{\mathbb{N}}

\newcommand{\Rbb}{\mathbb{R}}

\newcommand{\Kcal}{\mathcal{K}}

\def\Bn{{\bf n}}

\def\Br{{\bf r}}

\def\Bx{{\bf x}}
\def\By{{\bf y}}

\def\BG{{\bf G}}
\def\BH{{\bf H}}

\def\BP{{\bf P}}

\newcommand{\Ge}{\epsilon}

\newcommand{\Gvf}{\varphi}

\newcommand{\Gl}{\lambda}

\newcommand{\Gs}{\sigma}

\newcommand{\GD}{\Delta}

\newcommand{\GO}{\Omega}

\newcommand{\beq}{\begin{equation}}
\newcommand{\eeq}{\end{equation}}

\numberwithin{equation}{section}
\numberwithin{figure}{section}

\begin{document}

\title{Weyl's law for the eigenvalues of the Neumann--Poincar\'e operators in three dimensions:  
Willmore energy and surface geometry}

\author{Yoshihisa Miyanishi \thanks{Center for Mathematical Modeling and Data Science, Osaka University, Osaka 560-8531, Japan. 
Email: {\tt miyanishi@sigmath.es.osaka-u.ac.jp}.}}
\date{}
\maketitle

\begin{abstract}
We deduce eigenvalue asymptotics of the Neumann--Poincar\'e operators in three dimensions. 
The region $\GO$ is  $C^{2, \alpha}$ ($\alpha>0$) bounded in $\Rbb^3$ and the Neumann--Poincar\'e operator $\Kcal_{\p\GO} : L^2(\partial \Omega) \rightarrow L^2(\partial \Omega) $ is defined by     
$$
\Kcal_{\p\GO}[\psi](\Bx) := \frac{1}{4\pi} \int_{\partial \Omega} \frac{\langle \By-\Bx, \Bn(\By) \rangle}{|\Bx-\By|^3} \psi(\By)\; dS_{\By}
$$
where $dS_{\By}$ is the surface element and $\Bn({\By})$ is the outer normal vector on $\partial \Omega$. 
Then the ordering eigenvalues of the Neumann--Poincar\'e operator $\lambda_j (\Kcal_{\p\GO})$ satisfy 
$$
|\lambda_j(\Kcal_{\p\GO})| \sim \Big\{\frac{3W(\p\GO) - 2\pi \chi(\p\GO)}{128 \pi} \Big\}^{1/2} j^{-1/2}\quad \text{as}\ j \rightarrow \infty. 
$$
Here $W(\p\GO)$ and $\chi(\p\GO)$ denote, respectively, the {\it Willmore energy}  and the {\it Euler charateristic}  of the boundary surface $\p\GO$. This formula is the so-called Weyl's law for eigenvalue problems of Neumann--Poincar\'e operators.
\end{abstract}

\noindent{\footnotesize {\bf 2010 AMS subject classifications}. 47A75 (primary), 58J50 (secondary)}

\noindent{\footnotesize {\bf Key words}. Neumann-Poincar\'e operator, Eigenvalues, Weyl's law, Pseudo-differential operators, Willmore energy}


\section{Introduction and Results}
\par\hspace{5mm} 
The Neumann--Poincar\'e (abbreviated by NP) operator is a boundary integral operator which appears naturally when solving classical boundary value problems using layer potentials. Its study (for the Laplace operator) goes back to C. Neumann \cite{Neumann-87} and H. Poincar\'e \cite{Poincare-AM-87} as the name of the operator suggests. If the boundary of the domain, on which the NP operator is defined, is $C^{1, \alpha}$ smooth, then the NP operator 
is compact. Thus the Fredholm integral equation, which appears when solving Dirichlet or Neumann problems, can be solved using the Fredholm index theory \cite{Fredholm-03}. If the domain has corners, the NP operator is not any more a compact operator, but a singular integral operator. The solvability of the corresponding integral equation was established in \cite{Verch-JFA-84}.

Regarding spectral properties of the NP operator, it is proved in \cite{KPS} that the NP operator can be realized as a self-adjoint operator by introducing a new inner product on the $H^{-1/2}$-space (see also \cite{KKLSY-JLMS-16}), and so the NP spectrum consists of continuous spectrum and discrete spectrum (and possibly the limit points of discrete spectrum). If the domain has corners, the corresponding NP operator may exhibit a continuous spectrum (as well as eigenvalues). For recent development in this direction we refer to \cite{HKL-AIHP-17, KLY, PP-JAM-14, PP-arXiv}. If the domain has the smooth boundary, then the spectrum consists of eigenvalues converging to $0$. We refer to \cite{AKM2, Miyanishi:2015aa} for progress on the convergence rate of NP eigenvalues in two dimensions. However the satisfactory answers of decay rates in three dimensions were less-known even for smooth cases 
since it uses the smoothness of the kernel of NP operators in two dimensions \cite{Miyanishi:2015aa}. 
With this in mind, the purpose of this paper is to prove the so-called ``Weyl law'' which is the asymptotic behavior of NP eigenvalues in three dimensions.

To state the result in a precise manner, let $\Omega$ be a $C^{1, \alpha}$ bounded region in $\Rbb^3$. The NP operator $\Kcal_{\p \GO} : L^2{(\p\GO)} \rightarrow L^2{(\p\GO)}$ is defined by 
\beq\label{definition of NP operators}
\quad \Kcal_{\p \GO}[\psi](\Bx) := \frac{1}{4\pi} \int_{\partial \Omega} \frac{\langle \By-\Bx, \Bn(\By) \rangle}{|\Bx-\By|^3} \psi(\By)\; dS_{\By}   
\eeq 
where $dS_{\By}$ is the surface element and $\Bn(\By)$ is the outer normal vector on $\partial \Omega$.  
As described above, we know that $\Kcal_{\p \GO}$ is a compact operator on $L^2(\partial \Omega)$ and its eigenvalues consist of at most countable numbers, with $0$ the only possible limit point. It is also known that the eigenvalues of the NP operator lie in the interval $(-1/2, 1/2]$ and the eigenvalue $1/2$ corresponds to constant eigenfunctions. We denote the set of NP eigenvalues counting multiplicities by 
\beq
\sigma(\Kcal_{\p \GO})
=\{\; \lambda_j(\Kcal_{\p \GO})\ \mid\ \frac{1}{2}=|\lambda_0(\Kcal_{\p \GO})| > |\lambda_1(\Kcal_{\p \GO})| \geq |\lambda_2(\Kcal_{\p \GO})| \geq\; \cdots \geq 0 \}. 
\eeq
Here our main purpose is to deduce the asymptotic behavior of NP eigenvalues by using the basic ingredients of surface geometry. To do this, we also define the Willmore energy  $W(\p\GO)$ by  
\beq\label{definition of Willmore energy}
W(\p\GO):=\int_{\p\GO} H^2(x)\; dS_{x}  
\eeq
where $H(x)$ is the mean curvature of the surface. Then we have: 
\begin{theorem}\label{main}
Let $\GO$ be a $C^{2, \alpha}$ bounded region with $\alpha>0$.  
Then 
\beq
|\lambda_j(\Kcal_{\p \GO})| \sim \Big\{\frac{3W(\p\GO) - 2\pi \chi(\p\GO)}{128 \pi} \Big\}^{1/2}  j^{-1/2}\quad \text{as}\ j\rightarrow \infty.  
\eeq
Here $W(\p\GO)$ and $\chi(\p\GO)$ denote, respectively, the Willmore energy and the Euler characteristic of the surface $\p\GO$.
\end{theorem}
Thus the NP operator has infinite rank \cite{KPS} and the decay rate of NP eigenvalues is $j^{-1/2}$ for $C^{2, \alpha}$ regions. Furthermore, the integral (\ref{definition of Willmore energy}) is specially interesting because it has the remarkable property of being invariant under  M\"obius transformations of ${\Rbb}^3$ \cite{Bl, Wh}. Thus we find that the asymptotic behavior of NP eigenvalues is  also M\"obius invariant since the Euler characteristic is topologically invariant. We will present some further facts and applications later (See section \ref{sec: applications}). 

To clarify the meaning of Theorem \ref{main}, let us consider the case $\p\GO=S^2$. It is proved by Poincar\'e \cite{Poincare-AM-87} that the NP eigenvalues on a two-dimensional sphere are $\frac{1}{2(2k+1)}$ for $k = 0, 1, 2\, \ldots$ and their multiplicities are $2k +1$ (see also \cite{AKMU}). So we may enumerate them as 
$$
{\underbrace{\frac{1}{2}}_{1}, \underbrace{\frac{1}{6}, \frac{1}{6}, \frac{1}{6}}_{3}, \underbrace{\frac{1}{10}, \frac{1}{10}, \frac{1}{10}, \frac{1}{10}, \frac{1}{10}}_{5}, \cdots, \underbrace{\frac{1}{2(2k+1)}, \cdots, \frac{1}{2(2k+1)}}_{2k+1}}, \cdots.
$$
It easily follows that the $j=k^2$th eigenvalue satisfies 
$$|\lambda_j(\Kcal_{S^2})|=\frac{1}{2(2k+1)} \sim \frac{1}{4} j^{-1/2}. $$
In contrast, one can verify these asymptotics from Theorem \ref{main}:    
$$|\lambda_j(\Kcal_{S^2})| \sim \Big\{\frac{3W(S^2) - 2\pi \chi(S^2)}{128 \pi} \Big\}^{1/2}  j^{-1/2}=\frac{1}{4} j^{-1/2}$$ since ${W(S^2) =4\pi}$ and $\chi(S^2)=2$.  This calculation, of course, is consistent with the asymptotic of the explicit eigenvalues. 

It is worth comparing  the decay rates for three dimensional NP eigenvalues obtained here with those for two  dimensional NP eigenvalues. For the two dimensional cases, it is well known that the eigenvalues of the integral operator $\Kcal_{\p \GO}$ are symmetric with respect to the origin \cite{BM, Sh}. The only exception is the eigenvalue $1/2$ corresponding to constant eigenfunctions. NP eigenvalues are invariant under M\"obius transformations \cite{MR0104934}. 
One of the main distinguished features is that the decay rates deeply depend on the smoothness of the boundary. Indeed,  we \cite{AKM2, Miyanishi:2015aa} proved the decay rate which depends on the smoothness of the $C^k$ smooth boundary 
$\partial \Omega$, that is, for any $\tau >-k+3/2$,   
$$
|\lambda^{\pm}_j(\Kcal_{\p \GO})| = o(j^{\tau}) \quad \text{as}\; j\rightarrow \infty, 
$$
where $o$ means the small order. Moreover for the analytic boundary, we have the exponential decay rate:  
$$
|\Gl^{\pm}_{j}(\Kcal_{\p \GO})| \le Ce^{-j\Ge} \quad \text{as}\; j\rightarrow \infty, 
$$
for any $j$. Here $\Ge$ is the modified Grauert radius of $\p\GO$ (See \cite{AKM2} for the precise statement). 

We also remark that $\Kcal_{\p \GO}$ is not self-adjoint on $L^2(\p\GO)$. 
This difficulty can be circumvented by restating from singular values into NP eigenvalues. To do this, this paper is organized as follows: In the next section we introduce 
the notations of surface geometry and state the relationships among  singular- and eigenvalues using Ky-Fan theorem and the triangular representation of NP operators. In section \ref{sec: symbol}, we provide the approximate pseudo-differential operators for NP operators. Then the relationships given in section \ref{sec: notations} yield the Weyl law of NP eigenvalues in section \ref{sec:  Weyl law}. 
The applications are provided in section \ref{sec: applications}. This paper ends with some discussions 
and a brief conclusion. 

\date{\textbf{Acknowledgement: }I am grateful to Prof. H. Kang and Prof. K. Ando for useful discussions on the early stages of this work. I would also like to thank the members of Inha university for their hospitality during my visit, when the main results of this paper were obtained.}

\section{Preliminaries and Notations}\label{sec: notations}

As preliminaries, we shall mention the notations of surface geometry and some results of Schatten class used in this paper. 

\subsection{Surface geometry}\label{sec: surface geometry}
Let $M$ be a two-dimensional surface without bounday and $\Br = \Br(s, t)$ be 
a regular parametrization of a surface in ${\Rbb}^3$, where $\Br$ is 
a smooth (at least $C^{2}$) vector-valued function of two variables. 
It is common to denote the partial derivatives of $\Br$ with respect to $s$ and $t$ 
by $\Br_s$ and $\Br_t$. The first fundamental form is the inner product 
on the tangent space of a surface in three-dimensional Euclidean space 
which is induced canonically from the dot product of ${\Rbb}^3$. 
We denote the first fundamental form the Roman numeral I:  
\beq
 \mathrm {I}=E ds^2 +2F ds\; dt +G dt^2. 
\eeq
Let $\Br(s, t)$ be a parametric surface. Then the inner product of two tangent vectors is
$$
{\displaystyle {\begin{aligned}&{}\quad \mathrm {I} (a\Br_{s}+b\Br_{t},c\Br_{s}+d\Br_{t})\\
&=ac ( \Br_{s}\cdot \Br_{s}) +(ad+bc)(\Br_{s}\cdot \Br_{t}) +bd( \Br_{t}\cdot \Br_{t}) \\
&=Eac+F(ad+bc)+Gbd,\end{aligned}}} 
$$
Thus 
\begin{align}
E=\Br_{s}\cdot \Br_{s},\ F=\Br_{s}\cdot \Br_{t},\ G=\Br_{t}\cdot \Br_{t} 
\end{align}
and hereafter we often write $E=g_{11},\ F=g_{12}=g_{21},\ G=g_{22}$. 

The second fundamental form of a general parametric surface is defined as follows: Regularity of the parametrization means that $\Br_s$ and $\Br_t$ are linearly independent for any $(s, t)$ in the domain of $\Br$, and hence span the tangent plane to $M$ at each point. Equivalently, the cross product $\Br_s \times \Br_t$ is a nonzero vector normal to the surface. The parametrization thus defines a field of unit normal vectors $\Bn$: 
$$
\Bn=\frac{\Br_s \times \Br_t}{|\Br_s \times \Br_t|}. 
$$
Then the second fundamental form is usually written as
\begin{align}
{\displaystyle \mathrm {I\!I} =L\,ds^{2}+2M\,ds\,dt+N\,dt^{2}. } 
\end{align}
Here the coefficients $L, M, N$ at a given point in the parametric $st$-plane are given by the projections of the second partial derivatives of $\Br$ at that point onto the normal line to $M$ and can be computed with the aid of the dot product as follows: 
$$
{\displaystyle L=\mathbf {r} _{ss}\cdot \mathbf {n} ,\quad M=\mathbf {r} _{st}\cdot \mathbf {n} ,\quad N=\mathbf {r} _{tt}\cdot \mathbf {n} .} 
$$
Under these notations, we can denote the Gaussian curvature $K$ and the mean curvature $H$ as  
$$
K:=\det A=\frac{LN-M^2}{EG-F^2}, \quad H:=\frac{1}{2}{\rm tr} A 
$$
where
$
A={\displaystyle {\begin{pmatrix}E&F\\F&G\end{pmatrix}}}^{-1} {\displaystyle {\begin{pmatrix}L&M\\M&N\end{pmatrix}}}
$ is the Weingarten matrix. 

We also recall the conformal parameters used in this paper:  
The parameter $(s, t)$ is said to be isothermal or conformal if the first fundamental form 
is written as
\begin{align}
 \mathrm {I}=e^{2\sigma}(ds^2+dt^2)\quad \text{(i.e. $E=G=e^{2\sigma}$,\ $F=0$)}
\end{align}
where $\sigma:=\sigma(s, t)$ is a $C^2$ function in $(s, t)$. It is emphasized that one can always take 
such coordinates without loss of regularity \cite{DK}. For the isothermal parameters, we find 
$$
K=\frac{LN-M^2}{E^2}, \quad H=\frac{L+N}{2E}  
$$
and the Gauss-Bonnet formula 
$$
\int_{M} K\; dS=2\pi \chi(M)
$$
holds true even for $C^2$ oriented compact surfaces \cite{Res}.

\subsection{Schatten class}
Let $K$ be a compact operator in separable Hilbert space $H$. We denote the singular values $\{s_j(K) \}$ as the family of eigenvalues of $(K^* K)^{1/2}$. Since the singular values are non-negative, we always assume the singular values are non-increasing:   
\beq
\sigma_{sing}(K)=\{\; s_j(K) \mid s_1(K) \geq s_2(K) \geq s_3(K) \geq\; \cdots\}. 
\eeq
Then the Schatten $p$-norm of $K$ is defined by 
$$\Vert K\Vert_{S^p} := {\rm{tr}}(K^*K)^{p/2}=\sum_{j=1}^{\infty} |s_j(K)|^{p}$$ 
and the $p$-th Schatten-class operator is a bounded linear operator on a Hilbert space 
with finite Schatten $p$-norm. Especially for $p=2$, $2$nd Schatten-class operator 
is so-called Hilbert-Schmidt class operator. Here we can show the convergence rate 
of singular values of the $p$-th Schatten class operators:  
\begin{lemma}\label{weak Schatten lemma}
If $K$ is in $p$-th Schatten class,  then ordered singular values satisfy 
$$
s_j(K)=o(j^{-1/p}) .
$$
\end{lemma}
\begin{proof}
$$
\Vert K\Vert_{S^p}^p=\sum_{j=1}^{\infty} |s_j(K)|^{p}<\infty.
$$
Thus, for all $\epsilon>0$, there exists $N\in \Nbb$ such that 
$$
(n-N)|s_n(K)|^{p} \leqq \sum_{N+1}^{n} |s_j(K)|^p<\epsilon,
$$
and hence, 
$$
n|s_n(K)|^{p}< 2\epsilon \quad \mbox{for all}\; n>2N.  
$$
Accordingly, $s_j(K)=o(j^{-1/p})$ as desired. 
\end{proof}
The class of operators $K$, which satisfy $s_j(K)=o(j^{-1/p})$, 
is called weakly $p$th-Scahtten class \cite{Si}. Thus if we know the class of operators, 
then the upper bounds of decay rates are obtained. We also have the precise relation 
between singular- and eigenvalues, which is convenient to derive NP eigenvalues 
from singular values:  
\begin{prop}\label{almost self-adjoint} Let $K$ be a compact operator. Assume the following {\rm(1)-(3):}  

{\rm(1)}\ $K-K^*$ is in Hilbert-Schmidt class. 

{\rm(2)}\ Eigenvalues of $K$ consist of real values. 

{\rm(3)}\ $s_j(K) \sim C j^{-1/2}\quad\text{as}\ j\rightarrow \infty$. 

Then $|\lambda_j(K)| \sim s_j(K) \sim Cj^{-1/2}\quad\text{as}\ j\rightarrow \infty$.
\end{prop}
We call the operator $K$, satisfying (1)-(3), ``{\it almost self-adjoint operator}''. 
\begin{proof}
As it is well known \cite{GK}, for any compact operator $K$, there are a compact normal operator $D$ 
and a compact quasinilpotent operator $V$, such that
\beq K = D + V\quad \text{and}\quad \sigma(D) = \sigma(K).\eeq
Since the spectrum $\sigma(D)=\sigma(K)$ is real, $D$ is a compact self-adjoint operator and  
$$
K^{*}=D+V^{*}. 
$$
Thus $K-K^{*}=V-V^*$ is in Hilbert-Schmidt class. The compact quasinilpotent operator $V$ with Hilbert-Schmidt imaginary part $\Im (V)=\frac{V-V^*}{2i}$ is also in Hilbert-Schimdt class (See e.g. \cite[Lemma 6.5.1]{Gil1} and \cite{Gil2}). 

From Ky-Fan theorem (see e.g. \cite{Dostanic} and references therein) and Lemma \ref{weak Schatten lemma}, the Hilbert-Schmidt operator $V$ is considered as a small perturbation of $K$. Thus   
 $$|\lambda_j(K)|=|\lambda_j(D)|=s_j(D) \sim s_j(K) \sim C j^{-1/2}\quad\text{as}\ j\rightarrow \infty. $$
\end{proof}

\section{NP operators as pseudo-differential operators}\label{sec: symbol}
We study the asymptotics of singular numbers of $\Psi$DO's which correspond to compact operators of the form (\ref{definition of NP operators}). To consider an integral operator as a $\Psi$DO is technically convenient also because the asymptotics can be expressed directly in terms of symbols. The starting point for us is to construct the approximate $\Psi$DO of the NP operator $\Kcal_{\p\Omega}$ modulo Hilbert-Schmidt operators. This is done by using local coordinates. 

\subsection{NP operators and approximations}\label{sec: double symbol} 
Let $\Bx_0 \in \p\GO$ and choose open neighborhoods $U_j$ ($j=1,2,3$) of $\Bx_0$ in $\p\GO$ so that  $U_1\cup {U_2} \subset U_3$ and $U_3$ has a local parametrization as in section \ref{sec: surface geometry}.  Let $\Gvf_j$ ($j=1,2$) be a smooth functions such that $\mbox{supp}\ \Gvf_1 \subset U_1$ and $\mbox{supp}\ \Gvf_2 \subset U_2$. Such a situation allows us to taking the local coordinates $(x(s, t)), y(s, t), z(s, t))$ and the surface element is given by $d\Gs(s, t)=|(x_s, y_s, z_s)\times(x_t, y_t, z_t)|ds \wedge dt$. 
Thus we obtain

\begin{align*}
\Gvf_2 \Kcal_{\p \GO} [\Gvf_1 f] (\Bx) := \Gvf_2(\Bx) \int_{\Rbb^2} \Big[ & \frac{(x(s_1, t_1)-x(s_2, t_2)) (y_s z_t-z_s y_t)(s_1, t_1)}{4\pi |\Bx-\By|^{3}} \\
+& \frac{(y(s_1,t_1)-y(s_2, t_2))(z_s x_t-x_s z_t)(s_1, t_1)}{4\pi |\Bx-\By|^{3}} \\
+& \frac{(z(s_1,t_1)-z(s_2, t_2))(x_s y_t-y_s x_t)(s_1, t_1)}{4\pi |\Bx-\By|^{3}} \Big]
\Gvf_1(\By) f(\By) ds_1 \wedge dt_1. \\
\end{align*}

Remarking that $dx=x_s ds+ x_t dt$ , $dy=y_s ds+ y_t dt$ and $dz=z_s ds+ z_t dt$ on local charts, 
the metric of the surface is denoted as 
\begin{align*} dx^2+dy^2+dz^2 
&=(x_s^2+y_s^2+z_s^2) ds^2 +2(x_s x_t+y_s y_t+z_s z_t )ds\ dt+(x_t^2+y_t^2+z_t^2) dt^2 \\
&=E ds^2 + 2F ds\ dt +G dt^2.     
\end{align*} 
Letting  $y'=(s_1, t_1)$ and $x'=(s_2, t_2)$, we also find that
\begin{align*}
&|\Bx-\By|^2 \\
=&(x(s_1, t_1)-x(s_2, t_2))^2+(y(s_1, t_1)-y(s_2, t_2))^2+(z(s_1, t_1)-z(s_2, t_2))^2 \\
=&\{ E(x') (s_1-s_2)^2+2 F(x') (s_1-s_2)(t_1-t_2)+G(x') (t_1-t_2)^2 \}(1+O(|x'- y'|)),  
\end{align*}
and so 
\begin{align*}
&\frac{1}{|\Bx-\By|^{3}} \\ 
=& \frac{1}{[E(x') (s_1-s_2)^2+2 F(x') (s_1-s_2)(t_1-t_2)+G(x') (t_1-t_2)^2]^{3/2}}
(1+ O(|x'-y'|))  \\
=:& K(x', x'-y')+ E_1(x', y'). 
\end{align*}
Here  
$$
K(x', x'-y')=\frac{1}{[E(x') (s_1-s_2)^2+2 F(x') (s_1-s_2)(t_1-t_2)+G(x') (t_1-t_2)^2]^{3/2}}
$$ 
and  
$$
|E_1(x', y')| \lesssim |(s_1, t_1)-(s_2, t_2)|^{-2}.
$$

Similarly the numerator of the integral kernel $\Gvf_2 \Kcal_{\p \GO}[\Gvf_1 f] (\Bx)$ 
can be denoted as 
\begin{align*}
&(x(s_1, t_1)-x(s_2, t_2)) (y_s z_t-z_s y_t)(s_1, t_1) +(y(s_1,t_1)-y(s_2, t_2))(z_s x_t-x_s z_t)(s_1, t_1)\\
&+(z(s_1,t_1)-z(s_2, t_2))(x_s y_t-y_s x_t)(s_1, t_1) \\
=&(x_s(s_1-s_2)+x_t(t_1-t_2)+\frac{1}{2}(s_1-s_2)^2 x_{ss}+x_{st}(s_1-s_2)(t_1-t_2)+\frac{1}{2}x_{tt}(t_1-t_2)^2)(y_s z_t-z_s y_t)(s_1, t_1) \\
&+(y_s(s_1-s_2)+y_t(t_1-t_2)+\frac{1}{2}(s_1-s_2)^2 y_{ss}+y_{st}(s_1-s_2)(t_1-t_2)+\frac{1}{2}y_{tt}(t_1-t_2)^2)(z_s x_t-x_s z_t)(s_1, t_1)\\
&+(z_s(s_1-s_2)+z_t(t_1-t_2)+\frac{1}{2}(s_1-s_2)^2 z_{ss}+z_{st}(s_1-s_2)(t_1-t_2)+\frac{1}{2}z_{tt}(t_1-t_2)^2)(x_s y_t-y_s x_t)(s_1, t_1) \\ 
&+O(|x'-y'|^{2+\alpha}) \\
=&\frac{1}{2}(s_1-s_2)^2 \{x_{ss}(y_s z_t-z_s y_t) +y_{ss}(z_s x_t-x_s z_t)+z_{ss}(x_s y_t-y_s x_t) \}(s_1, t_1) \\
&+ (s_1-s_2)(t_1-t_2) \{ x_{st}(y_s z_t-z_s y_t) +y_{st}(z_s x_t-x_s z_t)+z_{st}(x_s y_t-y_s x_t) \}(s_1, t_1) \\
&+\frac{1}{2}(t_1-t_2)^2 \{x_{tt}(y_s z_t-z_s y_t) +y_{tt}(z_s x_t-x_s z_t)+z_{tt}(x_s y_t-y_s x_t) \}(s_1, t_1) 
\\ 
&+O(|x'-y'|^{2+\alpha}) \\
=&\frac{1}{2}\{ L(y') (s_1-s_2)^2 +2M(y') (s_1-s_2)(t_1-t_2)+N(y') (t_1-t_2)^2 \} |(x_s, y_s, z_s)\times(x_t, y_t, z_t)|
(s_1, t_1) \\
&+O(|x'-y'|^{2+\alpha}) \\
=&\frac{1}{2}\{ L(x') (s_1-s_2)^2 +2M(x') (s_1-s_2)(t_1-t_2)+N(x') (t_1-t_2)^2 \} |(x_s, y_s, z_s)\times(x_t, y_t, z_t)|
(s_1, t_1) \\
&+O(|x'-y'|^{2+\alpha}).  
\end{align*}
Here we used the second fundamental form:   
$$
{\displaystyle \mathrm {I\!I}}=L ds^2 +2 M ds\ dt +N dt^2   
$$
and
$$L(x')-L(y')=O(|x'-y'|^{\alpha}),\ M(x')-M(y')=O(|x'-y'|^{\alpha}),\ N(x')-N(y')=O(|x'-y'|^{\alpha}).$$ 
Summarizing the dominator and the numerator of the kernel, 
\begin{align*}
\Gvf_2 \Kcal_{\p \GO}[\Gvf_1 f] (\Bx)& \\
=\frac{1}{8\pi} \Gvf_2(\Bx)\int_{\Rbb^2} 
\Big[&\{ L(s_1-s_2)^2 +2M(s_1-s_2)(t_1-t_2)+N(t_1-t_2)^2 \}   \\
&|(x_s, y_s, z_s)\times(x_t, y_t, z_t)|
+O(|x'-y'|^{2+\alpha}) \Big] \{K(x', x'-y')(1 + O(|x'-y'|) \}
\Gvf_1(\By) f(\By) ds_2 \wedge dt_2 \\
=\frac{1}{8\pi}\Gvf_2(\Bx) \int_{\Rbb^2} \Big[&\{ L(s_1-s_2)^2 +2M(s_1-s_2)(t_1-t_2)+N(t_1-t_2)^2 \}
\\ 
&K(x', x'-y') |(x_s, y_s, z_s)\times(x_t, y_t, z_t)| + E_2(x', y') \Big] \Gvf_1(\By) f(\By) ds_2 \wedge dt_2 \\
=\frac{1}{8\pi}\Gvf_2(\Bx) \int_{\Rbb^2} \Big[&\{ L(s_1-s_2)^2 +2M(s_1-s_2)(t_1-t_2)+N(t_1-t_2)^2 \} \\ 
&K(x', x'-y') \Big] \Gvf_1(\By) f(\By) d\sigma(s, t) +\BH[f](\Bx) 
\end{align*}
where $\BH[f](\Bx)=\frac{1}{8\pi}\Gvf_2(\Bx) \int_{\Rbb^2}E_2(x', y')\Gvf_1(\By) f(\By) ds_2 \wedge dt_2$ and $E(x', y')=O(|x'-y'|^{1-\alpha})$.

It follows that $E_2(x', y')$ is in $L^{2}_{\text{loc}}(\Rbb^4)$ and $\BH$ is in Hilbert-Schmidt class by Mercer's theorem (See e.g. \cite{CH, Si}).   

\subsection{Symbols of the approximate NP operators}
From the calculations in the previous subsection, the NP operator on local charts is denoted as  
\begin{align*}
\Gvf_2 \Kcal_{\p \GO}[\Gvf_1 f] (\Bx) &\equiv \Gvf_2 \BP[\Gvf_1 f](\Bx) \\
&:=\frac{1}{8\pi}\Gvf_2(\Bx)\int_{\Rbb^2} \Big[\{ L(s_1-s_2)^2 \\
&\quad\quad+2M(s_1-s_2)(t_1-t_2)+N(t_1-t_2)^2 \}K(x', x'-y') \Big] \Gvf_1(\By) f(\By) d\sigma(s, t)  
\end{align*}
where the notation $\equiv$ stands for modulo Hilbert-Schmidt operators.  
Let us denote $\BP[f](\Bx)$ as the pseudo-differential operator (See e.g. \cite{Shubin, Ta} for the details). The classical $\Psi $DO is defined as
$$
Op(\sigma)[f](x) = \frac{1}{(2\pi)^2}\int_{\Rbb^2}\int_{\Rbb^2} \sigma(x, \xi) e^{i(x-y)\cdot \xi} f (y) dx d\xi 
$$
where $\sigma(x, \xi)$ is called the symbol of $Op(\sigma)$. Write $(x, y)$ instead of $(s, t)$ as usual. 
Our purpose is to get the $\Psi$DO representation of the operator $P^g_{jk}$ :  
\beq\label{partial pseudo-differential}
P_{jk}^g[f](x') :=\frac{1}{8\pi}  \int_{\Rbb^2} {(x_j-y_j)(x_k-y_k)}K(x', x'-y') f(\By) d\sigma 
\eeq
where  
$$
K(x', x'-y') = 
\frac{1}{[g_{11}(x')(x_1-y_1)^2+2 g_{12}(x')(x_1-y_1)(x_2-y_2)+g_{22}(x')(x_2-y_2)^2]^{3/2}}. 
$$ 
We can calculate the principal symbol of $P_{jk}^g$ with the aid of the {\it surface Riesz transforms} $R^g_k$:    
\beq\label{symbol of surface Riesz trans}
R_k^g[f](x') =\frac{1}{2\pi}  \int_{\Rbb^2} {(x_k-y_k)}K(x', x'-y') f(\By) dy'.
\eeq
Here $R_k^g$ is, in fact, a homogeneous pseudo-differential operator \cite{AKM}.  Let us recall the symbol of $R_k^g$ for the reader's convenience: 
\begin{lemma}\label{symbolRG}
For $f\in  C_0^{\infty}(\Rbb^2)$
\beq
R_k^g[f](x') =\frac{1}{(2\pi)^2}   \int_{\Rbb^4}\frac{-i}{\sqrt{\det (g_{kl}(x))}}\frac{\sum_{l} g^{kl}(x')\xi_l}{\sqrt{\sum_{k, l} g^{kl}(x')\xi_k \xi_l}} e^{(x'-y')\cdot \xi} f(\By) dy' d\xi.
\eeq
\end{lemma}
\begin{proof}
The matrix (tensor) $\BG(x)=(g_{ij}(x))$ is symmetric and one can diagonalize via orthogonal matrices:  
\beq
P^{-1}(x) \BG(x) P(x)=
\begin{pmatrix}
\alpha^2(x) &  0 \\
 0 & \beta^2(x)
\end{pmatrix}
\eeq
and 
\beq
\begin{pmatrix}
\alpha^{-1}(x) &  0 \\
 0 & \beta^{-1}(x)
\end{pmatrix}
P^{-1}(x) \BG(x) P(x)
\begin{pmatrix}
\alpha^{-1}(x) &  0 \\
 0 & \beta^{-1}(x)
\end{pmatrix}=
\begin{pmatrix}
1 &  0 \\
0 & 1
\end{pmatrix}. 
\eeq
Since $\alpha(x)\beta(x)=\sqrt{\det (g_{kl}(x))}$,
putting $z_k=x_k-y_k$ $(k=1, 2)$ and 
$$
z=
\begin{pmatrix}
z_1 \\
z_2
\end{pmatrix}
=
P(x)
\begin{pmatrix}
\alpha^{-1}(x) &  0 \\
 0 & \beta^{-1}(x)
\end{pmatrix}
\begin{pmatrix}
\tilde{z_1} \\
\tilde{z_2}
\end{pmatrix}
=\widetilde{P(x)}\tilde{z}, 
$$
 (\ref{symbol of surface Riesz trans}) becomes 
\begin{align*}
R_k^g[f](x') &=\frac{1}{2\pi} \int_{\Rbb^2} \frac{[\widetilde{P(x')} \tilde{z}]_k}{|\tilde{z}|^3}\ f(\By) \frac{d\tilde{z}}{\alpha(x')\beta(x')} \\
&=\frac{1}{(2\pi)^2}\int_{\Rbb^2} \int_{\Rbb^2} \frac{-i}{\sqrt{\det (g_{kl}(x))}}  \frac{[\widetilde{P(x)} \widetilde{P(x)}^{T}  \xi]_k}{|\widetilde{P(x)}^{T}  \xi|} e^{i\xi'(x'-y')} f(\By)\; d\xi dy' \\
&=\frac{1}{(2\pi)^2}\int_{\Rbb^2} \int_{\Rbb^2} \frac{-i}{\sqrt{\det (g_{kl}(x'))}} 
\frac{\sum_{k} g^{kl}(x')\xi_l}{\sqrt{\sum_{j, k} g^{kl}(x)\xi_j \xi_k}}e^{i\xi'(x'-y')} f(\By)\; d\xi dy'. 
\end{align*}
as desired. 
\end{proof}

As a consequence of Lemma \ref{symbolRG}, we get the 
symbol of $P_{jk}^g$ as homogeneous pseudo-differential operators: 
\begin{lemma}\label{symbolP} 
For $f\in C_0^{\infty}(\Rbb^2)$, 
\beq
P_{jk}^g[f](x') =\frac{1}{(2\pi)^2}   \int_{\Rbb^4} \Big[\frac{(-1)^{i-j} \widehat{\xi_j} \widehat{\xi_k}}{4 \det (g_{ij}(x')) \{ \sqrt{\sum_{j, k} g^{jk}(x')\xi_j \xi_k}\}^3}\Big]e^{i(x'-y')\cdot \xi} f(\By) dy' d\xi
\eeq
Here $\widehat{\xi_1}=\xi_2$ and $\widehat{\xi_2}=\xi_1$. 
\end{lemma}
\begin{proof}
From Lemma \ref{symbolRG} and using the integration by parts as oscillatory integrals \cite{Shubin}, we have for all $f \in  C_0^{\infty}(\Rbb^2),$ 
\begin{align*}
P_{jk}^g[f](x') =& \frac{1}{4} R_k^g [ \sqrt{\det (g(x))}(x_j-y_j)f](x') \\
=&\frac{1}{(2\pi)^2}   \int_{\Rbb^4}{-i} \frac{\sum_{l} g^{kl}(x')\xi_l}{4 \sqrt{\sum_{k, l} g^{kl}(x')\xi_k \xi_l}}\frac{1}{i} \frac{\p}{\p\xi_j} e^{i(x'-y')\cdot \xi} f(\By) dy' d\xi \\
=&\frac{1}{(2\pi)^2}   \int_{\Rbb^4} \frac{\p}{\p\xi_j} \Big[\frac{\sum_{l} g^{kl}(x')\xi_l}{4 \sqrt{\sum_{k, l} g^{kl}(x')\xi_k \xi_l}}\Big] e^{i(x'-y')\cdot \xi} f(\By) dy' d\xi \\
=&\frac{1}{(2\pi)^2}   \int_{\Rbb^4} \Big[\Big\{\frac{g^{kj}(x')}{4 \sqrt{\sum_{k, l} g^{kl}(x')\xi_k \xi_l}}\Big\} \\
&\quad\quad\quad -\frac{1}{2}\Big\{\frac{\{\sum_{l} g^{kl}(x')\xi_l\}\{ 2\sum_{l} g^{jl}(x') \xi_l  \}}{4 \{\sqrt{\sum_{k, l} g^{kl}(x')\xi_k \xi_l}\}^{3}}\Big\} \Big] e^{i(x'-y')\cdot \xi} f(\By) dy' d\xi \\
=&\frac{1}{(2\pi)^2}   \int_{\Rbb^4} \Big[\frac{g^{kj}(x') (\sum_{k, l} g^{kl}(x')\xi_k \xi_l) 
-\{\sum_{l} g^{kl}(x')\xi_l \}\{ \sum_{l} g^{jl}(x') \xi_l  \} }{4 \{\sqrt{\sum_{k, l} g^{kl}(x')\xi_k \xi_l}\}^3}\Big] e^{i(x'-y')\cdot \xi} f(\By) dy' d\xi \\
=&\frac{1}{(2\pi)^2}   \int_{\Rbb^4} \Big[\frac{ \sum_{l, m} (g^{kj}(x')g^{lm}(x') -g^{kl}(x')g^{jm}(x')  ) \xi_l \xi_m}{4 \{\sqrt{\sum_{j, k} g^{jk}(x')\xi_j \xi_k}\}^3}\Big]e^{i(x'-y')\cdot \xi} f(\By) dy' d\xi \\
=&\frac{1}{(2\pi)^2}   \int_{\Rbb^4} \Big[\frac{(-1)^{i-j}(g^{11}g^{22} -g^{12}g^{21})\widehat{\xi_j} \widehat{\xi_k}}{4 \{\sqrt{\sum_{j, k} g^{jk}(x')\xi_j \xi_k}\}^3}\Big]e^{i(x'-y')\cdot \xi} f(\By) dy' d\xi \\
=&\frac{1}{(2\pi)^2}   \int_{\Rbb^4} \Big[\frac{(-1)^{i-j} \widehat{\xi_j} \widehat{\xi_k}}{4 \det (g_{ij}) \{ \sqrt{\sum_{j, k} g^{jk}(x')\xi_j \xi_k}\}^3}\Big]e^{i(x'-y')\cdot \xi} f(\By) dy' d\xi.
\end{align*}
as  desired. 
\end{proof}

Hence the principal symbol of $P_{jk}^g$ is a (strictly) homogeneous symbol of order $-1$ and the summation immediately yields the principal symbol of $\BP$: 

\begin{lemma}\label{symbolT}
Let $\p\GO$ be a bounded $C^{2, \alpha}$ surface. Then  
$$\BP \equiv Op\Big( \Big[    
\frac{L(x') \xi_2^2 -2M(x')\xi_1 \xi_2 +N(x')\xi_1^2}{4 \det (g_{ij}) \{ \sqrt{\sum_{j, k} g^{jk}(x')\xi_j \xi_k}\}^3} \Big]
\Big) \quad\mbox{modulo}\ \text{Hilbert-Schmidt operators}. 
$$
\end{lemma}
We remark that the above $\Psi$DO is defined even for $f \in {L^2}(\p\GO)$ since the localizations of $\Psi$DO on local coordinates coincide with the sum of (\ref{partial pseudo-differential}). 
    
\section{Weyl's law of compact pseudo-differential operators and NP operators}\label{sec:  Weyl law}
Let us introduce Weyl's law of singular values of the pseudo-differential operators with $C^{\alpha}$ smooth in $x$-variable. M. S. Birman and M. Z. Solomyak \cite{BS} showed the asymptotics under weak smoothness hypothesis both in the $x$- and $\xi$-variable  (See \cite{Dostanic, Grubb, RT} for recent progress). In our situation, we employ the results as the asymptotics of singular values of $-1$ homogeneous $\Psi$DO in two dimensions.: 
\begin{theorem}[\cite{Grubb} Theorem 2.1 and Theorem 2.5]\label{Grubb theorem}  
On a closed manifold $M$ of dimension $2$, let $P$ be defined in local coordinates 
from symbols $p(x, \xi)$ that are homogeneous in $\xi$ of degree $-1$. Assume that the symbols restricted to $\xi \in S^{m-1}=\{|\xi|=1\}$ are in $C(S^{m-1}, C^{\epsilon})$ for some $\epsilon$. 
Then 
$$
s_j(P) \sim C(\p\GO)^{1/2}j^{-1/2}   \quad  j\rightarrow \infty. 
$$
Here 
\beq
C(\p\GO)=\frac{1}{8\pi^2} \int_{S^{*}M} |\sigma_0(x, \xi)|^{2} dx\ d\xi,   
\eeq
$S^{*}M$ denotes the cosphere bundle and $\sigma_0$ is the principal symbol of $Op(\sigma(x, \xi))$. 
\end{theorem}
From Theorem \ref{Grubb theorem}, immediately we have the Weyl's law of singular values:   
\begin{lemma}\label{Weyl law singular values}
Let $\GO$ be a $C^{2, \alpha}$ bounded region. Then 
\beq
s_j(\Kcal_{\p \GO})\sim \Big\{\frac{3W(\p\GO) - 2\pi \chi(\p\GO)}{128 \pi} \Big\}^{1/2} j^{-1/2}\quad \text{as}\ j\rightarrow \infty 
\eeq
where $W(\p\GO)=\int_{\p\GO} H^2\; dS$ and $\chi(\p\GO)$ denotes, respectively, the Willmore energy and the Euler characteristic of the surface $\p\GO$.
\end{lemma}
\begin{proof}
In the preceding section, we proved that $\Kcal_{\p \GO}$ is a $\Psi$DO modulo Hilbert-Schmidt class:  
\beq
\Kcal_{\p \GO} \equiv Op\Big( \Big[    
\frac{L(x) \xi_2^2 -2M(x)\xi_1 \xi_2 +N(x)\xi_1^2}{4 \det (g_{ij}) \{ \sqrt{\sum_{j, k} g^{jk}(x)\xi_j \xi_k}\}^3} \Big]
\Big) \quad \text{modulo Hilbert-Schimidt operator}\ \BH.
\eeq
From Ky-Fan theorem \cite{Dostanic}, $\BH$ is considered as the small perturbation of the $\Psi$DO. 
Thus $s_j(\Kcal_{\p \GO})$ also satisfies 
$$s_j(\Kcal_{\p \GO}) \sim C(\p\GO) j^{-1/2}   \quad  j\rightarrow \infty. $$
To calculate the positive constant $C(\p\GO)$ in Theorem \ref{Grubb theorem}, we take the isothermal charts introduced in section \ref{sec: surface geometry}. The surface element is given by $dS_{\Bx}=E(x) dx$ and  
\begin{align*}
C(\p\GO)^2&=\frac{1}{8\pi^2}\int_{\p\GO} \int_{S^1} \Big[\frac{L(x) \xi_2^2 -2M(x)\xi_1 \xi_2 +N(x)\xi_1^2}{4 \det (g_{ij}) \{ \sqrt{\sum_{j, k} g^{jk}(x)\xi_j \xi_k}\}^3} \Big]^2\; d\xi dx \\
&=\frac{1}{8\pi^2} \int_{\p\GO} \int_{S^1}\Big[\frac{L(x) \cos^2 \theta -2M(x)\cos \theta \sin \theta + N(x) \sin^2 \theta}{4 E^2(x) E^{-3/2}(x)} \Big]^2\; d\xi dx \\
&=\frac{1}{128\pi^2} \int_{\p\GO} \int_{S^1} \frac{ (L(x) \cos^2 \theta -2M(x)\cos \theta \sin \theta + N(x) \sin^2 \theta)^2}{E(x)} \; d\xi dx \\
&=\frac{1}{128\pi^2} \int_{\p\GO} \frac{ (\frac{3\pi}{4}L^2(x)+\frac{3\pi}{4}N^2(x)+\pi M^2(x) + \frac{\pi}{2}L(x)N(x)) }{E(x)} \; dx \\
&=\frac{1}{128\pi^2} \int_{\p\GO} \frac{ (\frac{3\pi}{4}L^2(x)+\frac{3\pi}{4}N^2(x)+\pi (L(x)N(x)-E^2(x)K(x)) + \frac{\pi}{2}L(x)N(x)) }{E(x)} \; dx \\
&=\frac{3}{512 \pi} \int_{\p\GO} \Big[\left(\frac{L(x)+N(x)}{E(x)} \right)^2 -\frac{4}{3} K(x) \Big]{E(x)} \; dx \\
&=\frac{3}{512 \pi} \int_{\p\GO} 4H^2(x)\; dx - \frac{1}{64} \chi (\p\GO) \\ 
&=\frac{3W(\p\GO) - 2\pi \chi(\p\GO)}{128 \pi}.  
\end{align*}
Thus we have a Weyl's type formula for singular values of NP operators. 
\end{proof}

\begin{proof}[Proof of Theorem \ref{main}]
From Proposition \ref{almost self-adjoint}, we need to prove only the almost self-adjointness of $\Kcal_{\p\GO}$, that is, 
$$ \Kcal^*_{\p\GO}-\Kcal_{\p\GO} $$
is a Hilbert-Schimidt operator. 
This fact follows from that the approximate $\Psi$DO's of $\Kcal^*_{\p\GO}$ and $\Kcal_{\p\GO}$ 
have just the same symbol from the similar calculations in section \ref{sec: symbol}.  
\end{proof}
\section{Applications}\label{sec: applications} 
Theorem \ref{main} states that the NP operators in three dimensions have infinite rank. 
Regarding the rank of NP operators, 
Khavinson-Putinar-Shapiro \cite{KPS} propose a question:  
 
``{\it The disk is the only planar domain for which the NP operator has finite rank. 
It is not known whether there are such domains in higher dimensions.}'' 

Our answer to this question is summarized as the following:
\par 
\begin{cor}[Finite-rank problem]
Let $\Omega$ be a bounded $C^{2, \alpha}$ {\rm($\alpha>0$)} region 
in $\Rbb^n$ {\rm($n=2, 3$)}. If the NP operator has finite rank, 
then
$$
n=2\ \mbox{and}\ \p\GO=S^1.
$$
Thus the finite rank NP operator is rank one. 
\end{cor}

Some results are also obtained from the known facts of the Willmore energy:  
The Willmore energy (\ref{definition of Willmore energy}) is known as the best thought of as a measure of `roundness', it is not hard to prove 
$$
W(\p\GO) \geq 4\pi, 
$$
with equality if and only if $\p\GO$ is an round sphere. 
For higher genus cases, F. C. Marques and A. Neves \cite{MN} proved 
the celebrated  ``Willmore conjecture''
$$
W(\p\GO) \geq 2\pi^2. 
$$
The equality is achieved by the torus of revolution whose generating circle
has radius $1$ and center at distance $\sqrt{2}$ from the axis of revolution:
$$
T^2_{\text{Clifford}}:=\{ ((\sqrt{2}+\cos u)\cos v, (\sqrt{2}+\cos u)\sin v, \sin u) \in {\Rbb}^3\; |\; (u, v) \in {\Rbb}^2  \}
$$
As results, we obtain the spectral geometry nature of NP eigenvalues: 
\begin{cor}\label{Isospectral property of sphere}
Let $\GO \subset \Rbb^3$ be a bounded region of class $C^{2, \alpha}$. Then 
$$\hspace{-6mm}|\lambda_j(\Kcal_{\p\GO})| \succsim \frac{1}{4} j^{-1/2}. $$
The minimum asymptotic is achieved if and only if $\partial \Omega=S^2$. Especially if 
$\sigma_p(\Kcal_{\p\GO})=\sigma_p(\Kcal_{S^2})$ then $\p\GO=S^2.$ 
\end{cor}
\begin{cor}\label{Isopectral property of Clifford Torus}
Let $\GO \subset \Rbb^3$ be a bounded region of class $C^{2, \alpha}$ with genus $g(\p\GO)\geq 1$.
Then $$|\lambda_j(\Kcal_{\p\GO})| \succsim \frac{\sqrt{3\pi}}{8} j^{-1/2}.$$ 
Especially if $\sigma_p(\p\GO)=\sigma_p(T^2_{\text{Clifford}})$ then $\p\GO\cong T^2_{\text{Clifford}}.$ 
Here $\cong$ means modulo M\"obius transforms.
\end{cor}
Any other properties of the Willmore energy can also be interpreted as the asymptotics of NP eigenvalues. 
For instance, Langevin and Rosenberg \cite{LR} showed that any knotted embedding of a torus in $\Rbb^3$
was bounded below by $8\pi$, namely, 
$$
{W(\p\GO) \geq 8\pi.} 
$$
In the terminology of NP eigenvalues, we have 
\begin{cor}\label{Decay estimate for knotted tori}
Let $\GO \subset \Rbb^3$ be a bounded region of class $C^{2, \alpha}$ and $\p\GO$ be a knot torus. 
Then
$$\hspace{-6mm}|\lambda_j(\Kcal_{\p\GO})| \succsim \frac{\sqrt{3}}{4} j^{-1/2}.$$
\end{cor}
 
As the last applicaltion, let us consider plasomonic eigenvalues (See e.g. \cite{Grieser} and references therein). A real number $\Ge$ is called a {\it plasmonic eigenvalue} if the following problem admits 
a solution $u$ in the space $H^1(\Rbb^3)$:
\beq\label{plasmon}
\begin{cases}
\GD u =0 \quad\quad\quad\quad\, &\text{in}\  {\Rbb}^3\backslash \p\GO ,\\
u|_{-}=u|_{+} \quad\quad\quad\ &\text{on}\ \p\GO ,\\
\Ge\p_n u|_{-}= -\p_n u|_{+} \quad\ \, &\text{on}\ \p\GO.
\end{cases}
\eeq
where the subscript $\pm$ on the left-hand side respectively denotes the limit (to $\p\GO$) from the outside and inside of $\GO$.
The well-known relation \cite{AKMU} among the plasmonic eigenvalue $\Ge$ and 
the NP eigenvalue $\Gl$ gives 
\beq\label{plasmonic eigenvalues}
|\epsilon_j - 1| =\Big|\frac{-2\lambda_j}{\lambda_j-1/2}\Big|\sim |4\lambda_j | \sim  \Big\{\frac{3W(\p\GO) - 2\pi \chi(\p\GO)}{8 \pi} \Big\}^{1/2}  j^{-1/2}. 
\eeq
Hence the plasmonic eigenvalues consist of the sequence with $1$ as the limit, and the (R.H.S.) of \eqnref{plasmonic eigenvalues} gives its converging rate. 
\section{Conclusion and discussions}
We discussed about the Weyl's law of NP eigenvalues. It depends on the Willmore energy and the Euler characteristics. However Theorem \ref{main} holds only for $C^{2, \alpha}$ smooth surfaces, while $\Kcal_{\p\GO}$ is compact for $C^{1, \alpha}$ surfaces. This fact indicates that the Weyl's law on surfaces having only $C^{1, \alpha}$ smoothness is probably changed to be the interpolation between $C^{0, 1}$ and $C^{2, \alpha}$. 

Moreover we don't know the asymptotics of {\it signed} NP eigenvalues:  
\beq
\lambda_1^{+}>\lambda_2^{+}>\cdots >0 >\cdots > \lambda_{2}^{-}> \lambda_{1}^{-}. 
\eeq
When we denote the {\it signed} Browder-G\r{a}rding density \cite{Andreev} as 
\beq
C_{\pm}(\p\GO)=\frac{1}{8\pi^2} \int_{S^{*}\p\GO}\Big[ \frac{L(x') \xi_2^2 -2M(x')\xi_1 \xi_2 +N(x')\xi_1^2}{4 \det (g_{ij}) \{ \sqrt{\sum_{j, k} g^{jk}(x')\xi_j \xi_k}\}^3} \Big]_{\pm}^2  dx\ d\xi  
\eeq
where the subscript $\pm$ denotes the positive and negative part respectively, we believe that  
\begin{align*}
\lambda_j^{\pm} \sim \pm C_{\pm}^{1/2}(\p\GO) j^{-1/2}\quad \text{as}\ j\rightarrow \infty.  
\end{align*}
If so, for the case of ellipsoids, $C_{-}(\p\GO)=0$ and negative NP eigenvalues decay faster than $j^{-1/2}$ \cite{Ah2, AA, Ma, Ri2}.
For a torus, $C_{-}(\p\GO)>0$ and infinitely many negative NP eigenvalues exist. 

We hope that the truth or falsehood of these problems will be established in the recent future.

\end{document}